\documentclass[11pt]{amsart}
\usepackage[mathscr]{eucal}
\usepackage{palatino, mathpazo, amsfonts, mathrsfs, amscd}
\usepackage[all]{xy}
\usepackage{url}
\usepackage{amssymb, amsmath, amsthm}
\usepackage{stackrel}
\usepackage{bm}

\usepackage{tikz-cd}

\newtheorem{theorem}{Theorem}[section]
\newtheorem{lemma}[theorem]{Lemma}

\newtheorem{corollary}[theorem]{Corollary}

\theoremstyle{remark}
\newtheorem{remark}[theorem]{Remark}

\newtheorem{definition}[theorem]{Definition}
\newtheorem{notation}[theorem]{Notation}

\numberwithin{equation}{subsection}
\usepackage{todonotes}

\makeatletter
\def\imod#1{\allowbreak\mkern10mu({\operator@font mod}\,\,#1)}
\makeatother

\newcommand{\cY}{\mathcal{Y}}
\newcommand{\cC}{\mathcal{C}}
\newcommand{\cX}{\mathcal{X}}
\newcommand{\cH}{\mathcal{H}}
\newcommand{\cL}{\mathcal{L}}

\newcommand{\cD}{\mathcal{D}}
\newcommand{\cZ}{\mathcal{Z}}

\newcommand{\CC}{\mathbb{C}}
\newcommand{\ZZ}{\mathbb{Z}}

\newcommand{\PP}{\mathbb{P}}
\newcommand{\QQ}{\mathbb{Q}}

\newcommand{\RR}{\mathbb{R}}

\newcommand{\sO}{\mathscr{O}}

\newcommand{\sMbar}{\overline{\mathscr{M}}}

\newcommand{\bq}{\mathbf{q}}

\newcommand{\bt}{\bm{t}}

\newcommand{\btau}{\bm{\tau}}
\newcommand{\ii}{\mathbb{1}}
\newcommand{\iI}{\mathbb{I}}
\newcommand{\jj}{\mathfrak{j}}

\newcommand{\bv}[1]{\mathbf{#1}}

\newcommand{\gr}{\textbf{Gr}}

\DeclareMathOperator{\age}{age}
\DeclareMathOperator{\diag}{diag}

\newcommand{\br}[1]{\left\langle#1\right\rangle}  

\title[Quantum Cohomology of Blowups]{Quantum Cohomology of Toric Blowups and Landau--Ginzburg Correspondences}

\author{Pedro Acosta}
\address{P.~Acosta, Department of Mathematics, University of Michigan,
Ann Arbor, MI 48109-1043, U.S.A.}
\email{peacosta@umich.edu}

\author{Mark Shoemaker}
\address{M.~Shoemaker, Department of Mathematics, University of Utah,
Salt Lake City, UT 84112-0090, U.S.A.}
\email{markshoe@math.utah.edu}

\begin{document}

\begin{abstract}
We establish a genus zero correspondence between the equivariant Gromov--Witten theory of the Deligne--Mumford stack $[\CC^N/G]$ and its blowup at the origin.  The relationship generalizes the crepant transformation conjecture of Coates--Iritani--Tseng and Coates--Ruan to the discrepant (non-crepant) setting using asymptotic expansion.  Using this result together with quantum Serre duality and the MLK correspondence we prove LG/Fano and LG/general type correspondences for hypersurfaces.
\end{abstract}

\maketitle
{\small
\tableofcontents}

\section{Introduction}

Given a birational map $f: \cY \dasharrow \cX$ between smooth complex varieties (or orbifolds), 
it is natural to ask if there exists a relationship between the Gromov--Witten theory of $\cX$ and $\cY$.  This question has a long history and has been studied by many authors (see for instance \cite{Ga,  Hu, BG, CIT, CR, HH,Ru}).
 For instance Hu proved in the symplectic setting that if $f$ is the blowup of a smooth point or a curve, many of the Gromov--Witten invariants remain unchanged.  
 
 In general however, the Gromov--Witten theory of $\cX$ and $\cY$ will \emph{not} be equal, and so the question remains whether or not it is possible to relate the two theories in any systematic manner.  In this paper we propose such a correspondence in terms of asymptotic expansion, and provide a proof in the case where $f: \cY \dasharrow \cX$ is a blowup of $[\CC^N/G]$ at the origin.
\newpage
\subsection{Asymptotic expansion}
\subsubsection{The crepant case}
A key insight comes from mirror symmetry, which suggests that for certain \emph{$K$-equivalent} varieties, the relationship between their Gromov--Witten theories is given by analytic continuation and quantization.  To be more precise, generating functions of Gromov--Witten invariants of $\cX$ should be related to generating functions of Gromov--Witten invariants of $\cY$ by analytic continuation and a (possibly quantized) symplectic transformation.  This relationship has been dubbed the crepant resolution (or more generally the crepant transformation) conjecture, and may be viewed as an instance of the McKay correspondence (see \cite{Re}) for Gromov--Witten theory.
The conjecture was proposed in various levels of generality by Li-Ruan \cite{LR}, Bryan--Graber \cite{BG}, Coates--Corti--Iritani--Tseng \cite{CIT}, and Coates--Ruan \cite{CR}.  
At this point the conjecture is well studied, and has been verified for a wide class of examples.  For genus zero correspondences see for instance  \cite{BG, CIT, CIJ, LLQW, GW} and for higher genus see \cite{CI, BCR, ILLW, Zhou}.

\subsubsection{The general case}

An important ingredient in the crepant transformation conjecture is the analytic continuation of generating functions of the respective invariants.  Thus a necessary component in this conjecture is that these generating functions, a priori given as power series in certain formal variables, \emph{are in fact analytic}, at least in some of the variables.  

This analyticity, so crucial in the above conjecture, is the first thing to fail for a general birational map $f: \cY \dasharrow \cX$.  In many cases one generating function will be analytic with an essential singularity at infinity, while the other generating function has radius of convergence equal to zero.  A solution to this obstacle was  discovered  by the first author in \cite{A} in the context of Landau--Ginzburg correspondences (see Section~\ref{sss:lgc}).  It was realized that in this context analytic continuation should be replaced by power series asymptotic expansion (see Section~\ref{the correspondence for toric blowups}).  

Our first main theorem concerns the case where $f: \cY \dasharrow \cX$ is a weighted blow-up of $\cX = [\CC^N/G]$ at the origin.  The precise assumptions on $\cX$ and $f$ are given in Section~\ref{s:BU}.  In the following, let $|\cX|$ denote the coarse space of $\cX$, and let $|f|: \cY \to |\cX|$ denote the map induced by $f$.  The discrepancy of $|f|$ is defined to be the coefficient of the exceptional divisor in $K_{\cY} - f^*(K_{|\cX|})$.
\begin{theorem}[see Theorems~\ref{theorem_asymptotic_correspondence} and~\ref{theorem_asymptotic_correspondence2}]
When $|f|: \cY \to |\cX|$ has positive discrepancy, the quantum cohomology of $\cY$ fully determines the quantum cohomology of $\cX$.  When $|f|: \cY \to |\cX|$ has negative discrepancy, the quantum cohomology of $\cX$ fully determines the quantum cohomology of $\cY$. 
\end{theorem}

The goal of this paper is to introduce the tool of asymptotic expansion as a way of relating generating functions of Gromov--Witten invariants of birational spaces.  We restrict ourselves here to the specific case described above for two reasons.  First, it allows us to state and prove the correspondence in a simple situation with (relatively) light notation.  Second, it is exactly the setting of relevance for Landau--Ginzburg correspondences, which we prove as an application of Theorems~\ref{theorem_asymptotic_correspondence} and~\ref{theorem_asymptotic_correspondence2}, and explain below.

We will prove a general version of the above theorem in the context of toric wall-crossing in a forthcoming paper.

\subsubsection{Landau--Ginzburg correspondences}\label{sss:lgc}

Given a non-degenerate quasi-homogeneous polynomial $W = W(X_1, \ldots, X_N)$ and an \emph{admissible} group $G$ \cite{FJR1}, one can define the corresponding \emph{FJRW} invariants, which may be viewed as the analogue of Gromov--Witten invariants for the singular space $\{W = 0\} \subset [\CC^N/G]$.  They are given as integrals over a cover of the moduli space of stable curves, and obey many of the same axioms as Gromov--Witten invariants.  

One may also view the pair $(W, G)$ as defining a hypersurface $\cZ: = \{W = 0\} \subset \PP(G)$, where $\PP(G)$ denotes a suitable quotient of weighted projective space.
Again inspired by mirror symmetry, the Landau--Ginzburg/Calabi--Yau correspondence is a conjectural relationship between the FJRW invariants of the pair $(W, G)$, and the Gromov--Witten theory of $\cZ$ in the case when $\cZ$ is a Calabi--Yau variety.  It was first verified for the quintic hypersurface in $\PP^4$  in genus zero \cite{ChR}, and has since been extended to all cases where $W$ is a Fermat polynomial and $G \leq SL_{N}(\CC)$ \cite{LPS}.

In \cite{A}, the first author extended this correspondence to the case where $\cZ$ was either a Fano or general type hypersurface in weighted projective space. Our second main theorem generalizes this to the case where $W$ is a Fermat polynomial: \[W = X_1^{d/c_1} + \cdots + X_N^{d/c_N}\] with $\gcd(c_1, \ldots, c_N)  = 1$,  and $W$ gives a section of a line bundle $\cL$ pulled back from the coarse space of $\PP(G)$.  The condition that $\sum_{j=1}^N c_j = d$ is exactly the condition necessary to guarantee that $\cZ$ is Calabi--Yau.  We deal here with the other cases.

\begin{theorem}[See Theorems~\ref{LG/Fano asymptotic correspondence} and~\ref{LG/gt asymptotic correspondence}]
If $\sum_{j=1}^N c_j - d > 0$, the quantum cohomology of the hypersurface $\cZ$ determines the quantum cohomology of the FJRW theory of $(W,G)$.  If $\sum_{j=1}^N c_j - d < 0$, the quantum cohomology of the FJRW theory of $(W,G)$ determines the quantum cohomology of $\cZ$.
%
\end{theorem}

\subsection{What is the correspondence?}

In the statements of the above theorems, we have been intentionally vague in saying how the quantum cohomology of one theory determines the other.  The precise statement is phrased in terms of generating functions of invariants.  
Consider the case where $f: \cY \dasharrow \cX$ is the blowup from above.
We consider certain generating functions, $I^\cY(q, z)$ and $I^\cX(t, z)$, of genus zero Gromov--Witten invariants of $\cX$ and $\cY$ (see Section~\ref{ss:gf}).  It is known that these functions fully determine the genus zero Gromov--Witten theory of $\cX$ and $\cY$ respectively.
In the case where $\cX$ and $\cY$ are $K$-equivalent, the crepant transformation conjecture is proven by showing that under the identification $t = q^{-d}$, there exists a unique linear transformation $L$ such that the analytic continuation of $L \cdot I^\cY(q, z)$ to a neighborhood of $t = 0$ yields the function $I^\cX(t, z)$.  

Consider now the case where $|f|: \cY \to |\cX|$ has positive discrepancy.  Then $I^\cX(t, z)$ will no longer be an analytic function.  The correct analogue of the above theorem is to show there exists a unique linear transformation $L$ such that the power series asymptotic expansion of $L \cdot I^\cY(q, z)$ recovers $I^\cX(t, z)$.  Because the power series asymptotic expansion of a function is uniquely determined (if it exists), this implies that $I^\cX(t, z)$ is uniquely determined by $I^\cY(q, z)$.  Because the $I$-functions fully determine the respective genus zero Gromov--Witten theories, this shows that the Gromov--Witten theory of $\cY$ fully determines the Gromov--Witten theory of $\cX$.  It is within this framework that all of our theorems are proven.

In a similar spirit, H. Iritani has also announced results relating the genus zero Gromov--Witten theory of $f: \cY \to \cX$ when $f$ is a blow-up of a toric variety.  His results are phrased however in terms of the quantum connection, and so not directly related to our statements on generating functions.  It will be interesting to understand the relationship between these respective frameworks.

\subsection{Acknowledgments}
The authors would like to thank H. Iritani for many useful conversations and for his talks on the crepant transformation conjecture.  They are also grateful to their advisor, Y. Ruan, for teaching them much of what they know about Gromov--Witten and FJRW theory. P.~A. would like to thank D. Ross for his interest in this work and for fruitful conversations. He also acknowledges the Algebraic Geometry group at the University of Utah for their hospitality during his visit in February of 2015. M.~S. would like to thank Y.-P. Lee for initially suggesting the MLK correspondence, on which the present paper relies.
M.~S.\ was partially supported by NSF RTG Grant DMS-1246989.

\section{An example}  
Consider the space $[\CC^3/\ZZ_2]$ with the diagonal action, and its resolution $\text{Tot}(\sO_{\PP^2}(-2))$.  In this section we will give an explicit computation of the relationship between the genus zero Gromov--Witten theory of these spaces.  This will serve to illustrate the general principle in a simple example. The $I$-function for the line bundle $\text{Tot}(\sO_{\PP^2}(-2))$ is given by
\[
I^{\text{Tot}(\sO_{\PP^2}(-2))}(q,z):=z\sum_{n\geq 0}q^{n+H/z}\frac{\prod_{l=0}^{2n-1}(-2(H+\lambda)-lz)}{\prod_{l=1}^n (H+lz)^3},
\]
where $H$ satisfies $H^3=0$ and $\lambda$ is the equivariant parameter of the torus action of $\CC^*$ on $\text{Tot}(\sO_{\PP^2}(-2))$.  A simple ratio test argument shows that up to a choice of branch cut for $\log(q)$, this series is holomorphic everywhere in the complex plane. From this it follows that the $I$-function cannot be extended holomorphically to the point $q=\infty$. It is still possible, however, to try to understand its asymptotic behavior as $q\rightarrow\infty$. The main claim of this example is that up to a linear transformation, the asymptotic behavior of $I^{\text{Tot}(\sO_{\PP^2}(-2))}$ as $q\rightarrow\infty$ is given by the $I$-function of $[\CC^3/\ZZ_2]$, i.e. there exists a unique linear transformation $L:H_{\CC^{\ast}}^{\ast}(\sO_{\PP^2}(-2))\longrightarrow H_{CR,\CC^{\ast}}^{\ast}([\CC^3/\ZZ_2])$ such that
\begin{equation}\label{example_asymptotics}
L\cdot I^{\text{Tot}(\sO_{\PP^2}(-2))}(q,z=1)\sim I^{[\CC^3/\ZZ_2]}(t=q^{-1/2},z=1)\quad\text{as $q\rightarrow\infty$.}
\end{equation}
Here $I^{[\CC^3/\ZZ_2]}$ is the $I$-function for the quotient $[\CC^3/\ZZ_2]$ given by the following formal series
\[
I^{[\CC^3/\ZZ_2]}(t,z):=z\sum_{k=0,1}\sum_{m\geq 0} t^{2m+k+2\lambda/z}\frac{\prod_{l=0}^{m-1}(-\lambda-(k/2+l)z)^3}{z^{2m+k}(2m+k)!}\ii_k,
\]
where $\ii_0$ and $\ii_1$ are the fundamental classes of the untwisted and twisted sectors of the inertia stack of $[\CC^3/\ZZ_2]$ respectively. The lesson we rescue from this result is that even though we no longer have an equivalence of local theories as in the crepant case, it is possible to obtain the genus zero theory of the singular quotient $[\CC^3/\ZZ_2]$ in terms of the asymptotics of the genus zero theory of its blow-up $\text{Tot}(\sO_{\PP^2}(-2))$.

To establish Equation \ref{example_asymptotics} we employ a technique known as \emph{Borel summation} (for a simple example see \cite[pp. 246-249]{Miller}), which consists of the following steps:

\textbf{Step 1:} \emph{Regularize the} $I$-\emph{function for the quotient stack} $[\CC^3/\ZZ_2]$. We define $I_{reg}(\tau)$ as
\[
I_{reg}(\tau):=\sum_{k=0,1}\sum_{m\geq 0} \tau^{m+k/2+\lambda}\frac{\prod_{l=0}^{m-1}(-\lambda-(k/2+l))^3}{\Gamma(1+m+k/2+\lambda)(2m+k)!}\ii_k.
\]
This series defines a holomorphic function for $\tau$ in the disk of radius $4$ centered at the origin. It is also straightforward to check that $I_{reg}(\tau)$ satisfies the following linear differential equation:
\begin{equation}\label{example_regularized_equation}
\left[\tau\left(\tau\frac{d}{d\tau}\right)^3+\left(2\tau\frac{d}{d\tau}-2\lambda\right)\left(2\tau\frac{d}{d\tau}-2\lambda-1\right)\left(\tau\frac{d}{d\tau}\right)\right]I_{reg}(\tau)=0.
\end{equation}
This differential operator has singular points for $\tau\in\{0,-4, \infty\}$ 
all of which are regular. It follows that $I_{reg}(\tau)$ can be analytically continued in the region of the complex plane defined by $|\arg\tau |<\pi$.

\textbf{Step 2:} \emph{Apply a Laplace transformation to $I_{reg}(\tau)$ and use Watson's lemma.} Since $I_{reg}(\tau)$ can be analytically continued to $\tau=\infty$, the following Laplace integral defines a holomorphic function:
\[
\iI(q):=q\mathcal{L}(I_{reg})(q):=q\int_0^{\infty}e^{-q\tau}I_{reg}(\tau)d\tau,
\]
where the ray on which we integrate avoids the singular point $\tau=-4$. As a consequence of Watson's lemma \cite{Miller}, we obtain the following asymptotic expansion for $\iI(q)$
\[
\begin{split}
\iI(q)\sim& \sum_{k=0,1}\sum_{m\geq 0} q^{-m-k/2-\lambda}\frac{\prod_{l=0}^{m-1}(-\lambda-(k/2+l))^3}{(2m+k)!}\ii_k\\
=&I^{[\CC^3/\ZZ_2]}(t=q^{-1/2},z=1)
\end{split}
\]
as $q\rightarrow\infty$. The upshot of this step is that we have constructed a holomorphic function $\iI(q)$ whose asymptotic expansion is given by the $I$-function of $[\CC^3/\ZZ_2]$.

\textbf{Step 3:} \emph{Show that $\iI(q)$ satisfies the Picard-Fuchs equation of $\text{Tot}(\sO_{\PP^2}(-2))$.} Expanding Equation \ref{example_regularized_equation} we obtain
\begin{multline*}
\left[(\tau^3+4\tau^2)\frac{d^3}{d\tau^3}+\left(3\tau^2+(10-8\lambda)\tau\right)\frac{d^2}{d\tau^2}\right.\\
+\left.\left(\tau+(2-2\lambda)(2-2\lambda-1)\right)\frac{d}{d\tau}\right]I_{reg}(\tau)=0.
\end{multline*}
Applying a Laplace transform to this yields
\begin{multline*}
-\frac{d^3}{dq^3}(q^2\iI(q))+\frac{d^2}{dq^2}(3q\iI(q)+4q^2\iI(q))
\\
-\frac{d}{dq}(\iI(q)+(10-8\lambda)q\iI(q))+(2-2\lambda)(1-2\lambda)\iI(q)=0,
\end{multline*}
(where we have assumed that $\Re(\lambda)\geq3$) which in turn is equivalent to
\[
\left[\left(q\frac{d}{dq}\right)^3-q\left(-2q\frac{d}{dq}-2\lambda\right)\left(-2q\frac{d}{dq}-2\lambda-1\right)\right]\iI(q)=0.
\]
This is precisely the Picard-Fuchs equation satisfied by $I^{\text{Tot}(\sO_{\PP^2}(-2))}(q,z=1)$. Since the components of $I^{\text{Tot}(\sO_{\PP^2}(-2))}$ are a complete set of solutions to this differential equation, there exists a unique linear transformation $L$ satisfying $L\cdot I^{\text{Tot}(\sO_{\PP^2}(-2))}(q,z=1)=\iI(q)$. Equation \ref{example_asymptotics} follows from this.
\begin{remark}
In order to compute the asymptotic expansion of $I^{\cY}$ we had to set $z$ equal to $1$ in the $I$-functions. We can recover the powers of $z$ in $I^{\cX}$ by means of the following procedure. Define a grading operator $\bv{Gr}$ by $\bv{Gr}(\lambda)=\lambda$, $\bv{Gr}(\ii_k)=\frac{3}{2}\ii_k$, then
\[
I^{[\CC^3/\ZZ_2]}(t,z)=z^{1-\bv{Gr}}z^{-\lambda}I^{[\CC^3/\ZZ_2]}(tz^{1/2},1).
\]
\end{remark}
\begin{remark}
Equation \ref{example_asymptotics} and Birkhoff factorization imply that the genus zero Gromov-Witten invariants of $[\CC^3/\ZZ_2]$ are completely determined by the genus zero Gromov-Witten invariants of $\text{Tot}(\sO_{\PP^2}(-2))$.
\end{remark}
\section{Gromov--Witten theory}\label{s:formal}

Here we review the basic definitions of (orbifold) Gromov--Witten theory and set notation.  For a reference see \cite{AGV} in the algebraic setting or \cite{ChenR3} in the symplectic.

\subsection{ Notation}
Let $\cX$ denote a smooth Deligne--Mumford stack with an equivariant action by a torus $T \cong (\CC^*)^r$.  Assume that the fixed point loci of $\cX$ is projective.  Let $H^*_{CR,T}(\cX) := H^*_{CR,T}(\cX; \CC)$ denote the equivariant Chen--Ruan orbifold cohomology of $\cX$ \cite{ChenR1}.  Recall that as a vector space, $H^*_{CR,T}(\cX) \cong H^*_T(I\cX)$ where $I\cX$ denotes the \emph{inertia stack}, parametrizing pairs $(x, g)$ where $x$ is a point in $\cX$ and $g \in G_x$ is an element of the isotropy group of $x$.  $I\cX$ is a disjoint union of connected components $I\cX = \coprod_{v \in V} \cX_v$ where each \emph{twisted sector} $\cX_v$ may be identified with a closed substack of $\cX$.  There is a distinguished component $\cX_{id}$ corresponding to the points $(x, id)$ which is isomorphic to $\cX$ itself.  We call this the \emph{untwisted sector} of $I\cX$.
Thus 
as a vector space
\[H^*_{CR,T}(\cX) \cong \bigoplus_{v \in V} H^*_T(\cX_v),\]
and by identifying the untwisted sector with $\cX$ itself we may view $H^*_T(\cX)$ as a summand of $H^*_{CR,T}(\cX)$.  The cohomology ring $H^*_{CR,T}(\cX)$ is a module over $R_T := H^*_T(\text{pt})$.

There is a natural involution map $inv: I\cX \to I\cX$ which sends $(x, g)$ to $(x, g^{-1})$.  We use this to define a pairing
\[ ( \alpha, \beta ) := \sum_{v \in V} \int_{\cX_v} \alpha \cup inv^*\beta
\] for $\alpha, \beta \in H^*_{CR,T}(\cX)$.

Given $\alpha_1, \ldots, \alpha_n$ elements of $H^*_{CR,T}(\cX)$ and integers $a_1, \ldots, a_n \in \ZZ_{\geq 0}$, we denote the \emph{Gromov--Witten invariant}
\[\br{ \psi^{a_1} \alpha_1, \ldots, \psi^{a_n} \alpha_n }_{g, n, d}^\cX := \int_{[\sMbar_{g, n}(\cX; d)]^{vir}} \prod_{i = 1}^n \psi_i^{a_i} ev_i^*(\alpha_i).
\]
Here $d \in NE(X)_\ZZ = NE(X) \cap H_2(|\cX|; \ZZ)$, $\sMbar_{g, n}(\cX; d)$ is the moduli space of stable maps of degree $d$ from  a genus $g$ orbi-curve with $n$ marked points into $\cX$, and $[ - ]^{vir}$ denotes the virtual class \cite{AGV}.  Finally, $\psi_i$ denotes the $\psi$-class at the $i$th marked point.  In the case where $\cX$ (and therefore possibly $\sMbar_{g, n}(\cX; d)$) is not proper, the above integral is defined via the (virtual) Atiyah--Bott localization formula \cite{GrP}.  In this case the invariant is defined only after inverting suitable equivariant parameters.

\subsection{Quantum cohomology}
Fix a basis $\{\phi_i\}_{i \in I}$ for $H^*_{CR,T}(\cX)$ such that for some subset $J \subset I$, $\{\phi_j\}_{j \in J}$ gives a basis of $H^2_T(\cX)$, the degree two classes in the untwisted sector.

We may express a general point in $H^*_{CR,T}(\cX)$ as $\bt = \sum_{i \in I} t^i \phi_i$.  It will be convenient to use double bracket notation, given $\alpha_j \in H^*_{CR,T}(\cX)$ and $a_j \in \ZZ_{\geq 0}$ as above, define
\[\br{\br{ \psi^{a_1} \alpha_1, \ldots, \psi^{a_n} \alpha_n }}^\cX:=
\sum_{d \in NE(X)_\ZZ} \sum_{k = 0}^\infty \frac{Q^d}{k!}
\br{ \psi^{a_1} \alpha_1, \ldots, \psi^{a_n} \alpha_n, \bt, \ldots, \bt
}_{0, n+k, d}^\cX\]
where on the right hand side we declare the degree zero Gromov--Witten invariants with one or two marked points to be zero, since the corresponding moduli spaces are empty.  In the above, the variables $Q^d$ are so-called Novikov variables, used to guarantee convergence of the sum.  

\begin{definition}
The quantum product $*_{\bt}: H^*_{CR,T}(\cX) \times H^*_{CR,T}(\cX) \to H^*_{CR,T}(\cX)$ is given by
\[( \alpha *_{\bt} \beta, \gamma) = 
\sum_{d \in NE(X)_\ZZ} \sum_{n = 0}^\infty \frac{1}{n!}
\br{ \br{ \alpha, \beta, \gamma}}^\cX
.\]
\end{definition}

By the divisor equation \cite{CK}, if we view the $t^i$ as formal variables, 
the specialization of the product obtained by setting $Q=1$ 
yields a well defined element of $H^*_{CR, T}(\cX)[[ \{t^i\}_{i \in I \setminus J}, \{e^{t_j}\}_{j \in J}]].$  
We apply this specialization in the sequel without further comment.
The quantum cohomology ring thus yields a formal deformation of $H^*_{CR, T}(\cX)$.

With this one may define the \emph{Dubrovin connection} 
\[\nabla_i: H^*_{CR, T}(\cX)[[ \{t^i\}_{i \in I \setminus J}, \{e^{t_j}\}_{j \in J}]]((z^{-1})) \to H^*_{CR, T}(\cX)[[ \{t^i\}_{i \in I \setminus J}, \{e^{t_j}\}_{j \in J}]]((z^{-1}))\] by \[
 \nabla_i :=  \frac{\partial}{\partial t^i} + z^{-1} \phi_i *_{\bt} - \hspace{.5 cm} \text{ for each $i \in I$.}
\]

\subsection{Generating functions}\label{ss:gf}

\begin{definition} The Givental $J$-function of $\cX$ is the cohomology-valued generating function of genus zero Gromov--Witten invariants given by
\[J^\cX(\bt, z) = z + \bt + \sum_{i \in I} 
\br{ \br{\frac{\phi_i}{z - \psi_1}}}^\cX \phi^i,
\]
where $\{\phi^i\}$ is the dual basis to $\{\phi_i\}$, and the expression $\frac{1}{z - \psi}$ is shorthand for the corresponding expansion in $1/z$.  
\end{definition}
The $J$-function also makes sense in the specialization $Q=1$.  Thus we can safely forget the Novikov variables.

The $J$-function arises naturally as a row of the solution matrix to the Dubrovin connection \cite{CK}.
For our purposes it is enough to note that the
quantum cohomology of $\cX$ is fully determined by the $J$-function.
The key point is that $J^\cX$ satisfies the system of partial differential equations
\[ z \frac{\partial}{\partial t^i}\frac{\partial}{\partial t^j} J^\cX(\bt, z)
= \sum_{k \in I}(\phi_i *_{\bt} \phi_j, \phi^k) \frac{\partial}{\partial t^k}J^\cX(\bt, z),
\]
which follows from the \emph{topological recursion relations}.


While the $J$-function for a given $\cX$ is often hard to calculate explicitly, we can instead work with so-called $I$-functions which have the advantage of being computable in many cases.  First, we upgrade the $J$-function to an endomorphism ${\bf J}^\cX(\bt, z): H^*_{CR, T}(\cX)((z^{-1}))[[\bt]] \to H^*_{CR, T}(\cX)((z^{-1}))[[\bt]]$ via:
\[{\bf J}^\cX(\bt, z): Y(\bt, z) \mapsto 
Y(\bt, z) + \sum_{i \in I} 
\br{ \br{ \frac{\phi_i}{z - \psi_1}, Y(\bt, z)
}}^\cX \phi^i
.\]
Note that $z{\bf J}^\cX(\bt, z)(1)$ recovers the original $J$-function via the string equation.

\begin{definition}\label{d:I function}
Let $q^1, \ldots, q^r$ be formal parameters.
An \emph{$I$-function} of $\cX$ is any cohomology-valued function of the form
\begin{equation}\label{Ifctn}
I^\cX(\bq, z) = z{\bf J}^\cX(\btau(\bq), z)(Y(\bq, z)),
\end{equation}
such that $Y(\bq, z) \in H^*_{CR, T}(\cX)[z][[\bq]]$ contains only \emph{positive} powers of $z$.
The map $\bq \mapsto \btau(\bq)$ is called the \emph{mirror map}.
\end{definition}

\begin{definition}
We say an $I$-function $I^\cX(\bq, z)$ is \emph{big} if there exist differential operators $\{P_i(z, z \frac{\partial}{\partial q^j})\}_{i \in I}$ which are polynomial in $z$ and $z \frac{\partial}{\partial q^j}$ such that 
\[z^{-1}P_i\left(I^\cX(\bq, z)\right) = \phi_i + O(\bq).\]
\end{definition}
The importance of big $I$-functions is explained in the following lemma.
\begin{lemma}\label{l:BF}
A big $I$-function $I^\cX(\bq, z)$ explicitly determines the $J$-function $J^\cX(\bq, z)$ (or rather its pullback under the mirror map).
\end{lemma}
\begin{proof}  This fact is contained in the proof of \cite[Theorem~5.15]{CIJ}.
The key is to use Birkhoff Factorization.  Consider the matrix, ${\bf I}^\cX(\bq, z)$, whose $i$th column is given by $z^{-1}P_i(z, z \frac{\partial}{\partial q^j}) I^\cX(\bq, z)$.  It follows from the topological recursion relations that 
$\frac{\partial}{\partial q^j}{\bf J}^\cX(\btau(\bq), z) = {\bf J}^\cX(\btau(\bq), z) (\btau^*\nabla)_j$, where $\btau^*\nabla$ is the pullback of the Dubrovin connection.  Therefore, 
\[{\bf I}^\cX(\bq, z) = {\bf J}^\cX(\btau(\bq), z) \circ \bf{Y}(\bq, z),\]
where the $i$th column of  $\bf{Y}(\bq, z)$ is defined to be $z^{-1}P_i(z, z (\btau^*\nabla)_j) Y^\cX(\bq, z)$.  Notice that the right hand side of the above equation is the Birkhoff factorization of the left hand side.  One may compute the right hand side recursively by expanding the above equation with respect to $\bq$.
\end{proof}

In what follows, we will work with $I$-functions, as it is generally difficult to obtain a closed form for the $J$-function.

\section{Weighted blowups of $[\CC^N/G]$}\label{s:BU}
In this section we introduce the specific spaces of interest and describe the corresponding $I$-functions.

We will restrict ourselves to the particular blowups of relevance to the Landau--Ginzburg correspondences of Section~\ref{s:FJRW}.  Although the proof holds for more general toric blowups with minimal modification, this simplifies the notation and exposition.

In particular, we consider birational spaces $\cY \dasharrow \cX$ where $\cX$ is a quotient stack  of the form $[\CC^N/G]$ and $\cY$ is a line bundle over weighted projective space obtained as a blowup of $\cX$.  In anticipation of our application to FJRW theory, we require that $G$ arises as a subgroup of the diagonal automorphisms of the Fermat polynomial 
\[W = X_1^{d/c_1} + \cdots + X_N^{d/c_N}\] where $\gcd(c_1, \ldots, c_N)  = 1$.
%
We assume that $G$ contains the distinguished automorphism 
\[\jj = \exp(2\pi i \diag (q_1, \ldots, q_N)),
\] 
where $q_j = c_j/d$ are the \emph{fractional weights} of $W$.  

\subsection{A toric description of the spaces}
Let $\bv{N} \cong \ZZ^N$ be a lattice, and let $\Sigma \subset \bv{N} \otimes \RR$ denote a fan such that $X_\Sigma$ is isomorphic to $\cX = [\CC^N/G]$.  $\Sigma$ contains a single maximal cone with generators $b_1, \ldots, b_N$, and $G \cong \bv{N}/\langle b_1, \ldots, b_N \rangle$.
Let 
\[\text{Box}(\Sigma) = \left\{ b' = \sum_{i=1}^N m_i b_i \in \bv{N} \; \vline\; m_i \in \QQ \cap [0,1) \text{ for } 1 \leq i \leq N\right\}.
\]
Note that the elements of $\text{Box}(\Sigma)$ are in bijection with those in $G$, and therefore index components of $I\cX$. 

The element $\jj \in G$ corresponds to the point $b' = \sum_{i=1}^N q_i b_i \in \text{Box}(\Sigma)$.  
Let $\Sigma '$ denote the star subdivision of $\Sigma$ obtained by adding the ray generated by $b'$.  Then $\cY := X_{\Sigma '}$ is isomorphic to the total space of the vector bundle $\sO_{\PP(G)}(-d)$ over the stack $\PP(G) := [\PP(c_1, \ldots, c_N)/\bar G]$ where $\bar G = G/\langle \jj \rangle$.  Thus $\cY$ gives a partial resolution of the coarse space of $\cX$.

One can easily check that the discrepancy of the toric morphism $|f|: \cY \to |\cX|$ is given by $\text{disc}(|f|) = \sum_{j=1}^N q_j - 1$.  This will play an important role in the what follows.

We endow $\cX$ and $\cY$ with compatible torus actions.  Let $T \cong \CC^*$ act on the coordinates of $\cX = [\CC^N/G]$ with weights $-c_1, \ldots, -c_N$.  On $\cY = \text{Tot}(\sO_{\PP(G)}(-d))$ this corresponds to a trivial action on the base $\PP(G)$ with a nontrivial action of weight $d$ in the fiber direction.  We let $\lambda$ denote the equivariant parameter of our torus action in $H^*_{CR,T}(\cX)$ and $H^*_{CR,T}(\cY)$.

%

\begin{notation}
For $g \in G$, we may express the action of $g$ on $\CC^N$ by \[\exp(2\pi i \diag (m_1(g), \ldots, m_N(g)))\] where $m_1(g), \ldots, m_N(g) \in \QQ \cap [0,1)$.  We call $m_j(g)$ the \emph{multiplicity} of $g$.  The \emph{age of $g$} is defined as $\sum_{j=1}^N m_j(g)$.
\end{notation}

Note that for any element $g \in G$, the fact that $G$ preserves the polynomial $W$ implies that $m_j(g)$ is either 0 or greater than or equal to $q_j$ for $1\leq j \leq N$.  This implies that $G$ splits as  $G \cong \langle \jj \rangle \oplus \bar{G}$.
Let us once and for all fix a splitting of $G$.  Choose generators $g_1, \ldots, g_k$ of $\bar{G}$ such that each $g_i$ fixes the first coordinate of $\CC^N$ (this is again possible due to the restrictions on the multiplicities $m_1(g_i)$) and $G \cong \langle \jj \rangle \oplus \langle g_1 \rangle \oplus \cdots \oplus \langle g_k \rangle$.   

\begin{notation}\label{n:S}
We let $\bar{G}$ denote the group generated by $g_1, \ldots, g_k$ from above.
\end{notation}

\subsection{I-functions}

\subsubsection{The $I$-function of $\cX$}

The inertia stack $I\cX$ is a disjoint union of components $\cX_g$ indexed by $g \in G$.  There is a natural choice of basis for the equivariant cohomology of $\cX$ given by $\{\ii_g\}_{g \in G}$, where $\ii_g$ is the fundamental class of $\cX_g$.  By abuse of notation we will also use $\ii_g$ to denote the fundamental class of the $g$th component of the inertia stack of $BG$.

\begin{notation}\label{n:Gcosets1}
It will be convenient to separate components of $I\cX$ according to $\bar G$ cosets.  Namely, 
\[ I\cX = \bigcup_{g \in \bar{G}} \bigcup_{0 \leq k \leq d-1} \cX_{\jj^k g}.\]  Let $H^*_g(\cX) : = H^*_T(\bigcup_{0 \leq k \leq d-1} \cX_{\jj^k g})$ denote the corresponding subspace in $H^*_{CR,T}(\cX)$.  We see that this has dimension $d$.
\end{notation}

Let $t^{g}$ denote the dual coordinate of $\tilde \ii_{g}$ for $g \in \bar{G}$, and let $t$ denote the dual coordinate of $\jj$.

We consider the $J$-function of $BG$, where the domain has been restricted to the span of $\{\ii_\jj\} \cup \{\ii_{g}\}_{g \in \bar{G}}$.  A simple computation involving $\psi$-classes on $\sMbar_{0, n}$ allows us to obtain an explicit formula (see \cite{LPS}, Lemmas~5.2 and~7.2):
\begin{align*}
J^{BG}(t, \bt, z) 
& = 
z\sum_{\bv{k}\in(\ZZ_{\geq 0})^{\bar{G}}}\prod_{g \in \bar{G}}\frac{(t^{g})^{k_g}}{z^{k_g}k_g!}
\sum_{k_0\geq 0}\frac{t^{k_0}}{z^{k_0}k_0!}\ii_{\jj^{k_0}\prod_g g^{k_g}}. 
\end{align*}

Using the twisted theory technology, one may alter $J^{BG}(t, \bt, z)$ by a \emph{hypergeometric modification} (see \cite{CCIT}) to obtain an $I$-function, $I^{\cX}(t, \bt, z)$ in the sense of Definition~\ref{d:I function}.
Let $a(\bv{k})^j=\sum_{s}k_g m_j(g)$.  Define the modification factor
\[
M(k_0,\bv{k}) := \prod_{j=1}^N \prod_{l=0}^{\lfloor k_0q_j+a(\bv{k})^j \rfloor-1}\Big(-c_j\lambda-(\langle k_0 q_j+a(\bv{k})^j  \rangle +l)z\Big)
\]
where $\langle -\rangle$ denotes the fractional part. Then $I^{\cX}(t, \bt, z)$ is defined as 
\begin{equation}
I^{\cX}(t, \bt,z)= z t^{d\lambda/z}
\sum_{\bv{k}\in(\ZZ_{\geq 0})^{\bar{G}}}\prod_{g \in \bar{G}}\frac{(t^{g})^{k_g}}{z^{k_g}k_g!}\sum_{k_0\geq 0}\frac{M(k_0,\bv{k})t^{k_0}}{z^{k_0}k_0!}\ii_{\jj^{k_0}\prod_g g^{k_g}}.
\end{equation}
The above modification factor is explained in \cite{CCIT}, where it is proven that $I^{\cX}(t, \bt, z)$ is a (big) $I$-function for $\cX$.  Using Gamma functions this simplifies to
\begin{align}
\nonumber
I^{\cX}(t, \bt,z) = &zt^{d\lambda/z}\sum_{\bv{k}\in(\ZZ_{\geq 0})^{\bar{G}}}\prod_{g \in \bar{G}}\frac{(t^{g})^{k_g}z^{(\age(g)-1)k_g}}{k_g!}\sum_{k_0\geq 0}\frac{t^{k_0}z^{k_0(\sum_j q_j - 1)}}{z^{\sum_j\langle k_0q_j+a(\bv{k})^j\rangle}k_0!}\\
  & \cdot \prod_{j=1}^N \frac{\Gamma(1-c_j\tfrac{\lambda}{z}- \langle k_0q_j+a(\bv{k})^j\rangle)}{\Gamma(1-c_j\tfrac{\lambda}{z}- k_0q_j-a(\bv{k})^j)}\ii_{\jj^{k_0}\prod_g g^{k_g}} \label{e:IX}
\end{align}

\subsubsection{The $I$-function of $\cY$}
The components of the inertia stack $I \cY$ are indexed by $\{g\}_{g \in G}$.  This follows from the facts that first each $c_j$ divides $d$ (and so components of $I\PP(c_1, \ldots, c_N)$ correspond to powers of $\jj$), and second that $G$ splits as $\langle \jj \rangle \oplus \bar{G}$.
Here the component $\cY_g$ of $I\cY$ is identified with the closed subset of $\cY$ obtained by setting $x_j = 0$ for all coordinates not fixed by $g$ (i.e. $m_j(g) \neq 0$).  An equivariant basis for the Chen--Ruan cohomology of $\cY$ is given by 
\[ \bigcup_{g \in G} \{ \tilde \ii_g, \tilde \ii_g H, \ldots, \tilde \ii_g H^{(\dim((\CC^N)^g) - 1)}\},
\]
where $\tilde \ii_g$ is the fundamental class of $\cY_g$ and $\tilde \ii_g H^k$ denotes the pullback of the $k$th power of the hyperplane class from the course space of $\cY_g$.  Here we use the convention that $\ii_g$ is zero if $\cY_g$ is empty (i.e., if the action of $g$ on $\CC^N$ fixes only the origin).

\begin{notation}\label{n:Gcosets2}
We may also separate components of $I\cY$ according to $\bar G$ cosets,
\[ I\cY = \bigcup_{g \in \bar{G}} \bigcup_{0 \leq k \leq d-1} \cY_{\jj^k g}.\]  One may check that the corresponding subspace
$H^*_g(\cY) := H^*_T( \bigcup_{0 \leq k \leq d-1} \cY_{\jj^k g})$ has dimension equal to $c_1 + \cdots + c_N$.
\end{notation}

We will also let $t^g$ denote the dual coordinate of $\tilde \ii_g$ for $g \in G$, and let $q$ denote the \emph{exponential} of the dual coordinate to $H$.  

An $I$-function for toric stacks is given in \cite{CCIT2}.  
Again using Gamma functions, a (big) $I$-function for $\cY$ takes the form
\begin{align}\nonumber
I^{\cY}(q,\bt,z) =&zq^{H/z} \sum_{\bv{k}\in(\ZZ_{\geq 0})^{\bar{G}}}\prod_{g \in \bar{G}}\frac{(t^{g})^{k_g}z^{(\age(g)-1)k_g}}{k_g!} \sum_{k_0\geq 0}\frac{q^{k_0/d}}{z^{k_0(\sum_j q_j - 1)+\sum_j\langle k_0q_j-a(\bv k)^j\rangle}} \\ \label{e:IY}
	&\cdot \frac{\Gamma(1-\tfrac{d(\lambda+H)}{z})}{\Gamma(1-k_0-\tfrac{d(\lambda+H)}{z})}\prod_{j=1}^N\frac{\Gamma(1+c_jH/z-\langle -k_0q_j+a(\bv{k})^j \rangle)}{\Gamma(1+c_jH/z+k_0q_j-a(\bv{k})^j)}\tilde \ii_{\jj^{-k_0}\prod_g g^{k_g}}
\end{align}
%

\section{The correspondence for toric blowups}\label{the correspondence for toric blowups}

\subsection{Asymptotic correspondence for positive discrepancy}
In this section we will state a correspondence of genus zero Gromov-Witten theories for the case when $|f|: \cY \to |\cX|$ has positive discrepancy.   In section \ref{correspondence_r_negative} we will discuss the other case.  Recall that the discrepancy of $|f|$ was equal to $\sum_{j = 1}^N q_j - 1$.
To simplify notation slightly we will work with the quantity $r = d \cdot \text{disc}(|f|)$,
 \[r:=\sum_{j=1}^N c_j -d.\] 

The idea behind the asymptotic correspondence of genus zero theories is to obtain the genus zero Gromov-Witten invariants of $\cX$ using the information provided by the genus zero Gromov-Witten invariants of $\cY$.
To be more specific, we show that the $I$-function of $\cX$ (and therefore its $J$-function) is completely determined by the $I$-function of $\cY$ through power series asymptotic expansion.
\begin{theorem}\label{theorem_asymptotic_correspondence}
Assume $\text{disc}(|f|)>0$. There exists a unique linear transformation $L:H_{CR,T}^{\ast}(\cY)\longrightarrow H_{CR,T}^{\ast}(\cX)$ such that 
\[
L\cdot I^{\cY}(q,\bt,z=1)\sim I^{\cX}(t=q^{-1/d}, \bt,z=1)\quad\text{as $q\rightarrow\infty$}.
\]
\end{theorem}
Theorem \ref{theorem_asymptotic_correspondence} has the following consequence:
\begin{corollary}
If  $\text{disc}(|f|)>0$, then the genus zero Gromov-Witten invariants of $\cX$ are completely determined by the genus zero Gromov-Witten invariants of $\cY$.
\end{corollary}
\begin{proof}
First note that $I^{\cY}$ may be viewed as a generating function of genus zero Gromov--Witten invariants.
$I^{\cX}(t=q^{-1/d}, \bt,z=1)$ is the power series asymptotic expansion of $L\cdot I^{\cY}(q,\bt,z=1)$ and it is therefore uniquely determined by $I^{\cY}$. We may recover $I^{\cX}(t=q^{-1/d}, \bt,z)$ from $I^{\cX}(t=q^{-1/d}, \bt,1)$ by means of the relations described in Section~\ref{z=1} below. Lastly, the Givental $J$-function of $\cX$ (and therefore, its genus zero theory) can be obtained from $I^{\cX}$ via Birkhoff factorization by Lemma~\ref{l:BF}.
\end{proof}
The proof that follows relies heavily on the recursive structure of the functions $I^\cY$ and $I^\cX$.  In particular we take advantage of an identification of the differential equations satisfied the respective $I$-functions to relate the functions themselves.
This is part of the larger theory of GKZ systems described by Gelfand--Kapranov--Zelevinsky in \cite{GKZ} and Adolphson in \cite{Ad}, although in what follows the relevant differential equations can easily be checked by hand.

\begin{notation}
In order to simplify the computations in the subsequent sections, we introduction the following notation. Define
\begin{multline*}
I^{\cY}_{\bv{k}}(q,z):=z\sum_{k_0\geq 0}\frac{q^{k_0/d+H/z}}{z^{k_0r/d+\sum_j\langle k_0q_j-a(\bv k)^j\rangle}} \frac{\Gamma(1-\tfrac{d(\lambda+H)}{z})}{\Gamma(1-k_0-\tfrac{d(\lambda+H)}{z})}\\
\cdot\prod_{j=1}^N\frac{\Gamma(1+c_jH/z-\langle -k_0q_j+a(\bv{k})^j \rangle)}{\Gamma(1+c_jH/z+k_0q_j-a(\bv{k})^j)}\tilde \ii_{\jj^{-k_0}\prod_g g^{k_g}}\quad\text{and}
\end{multline*}
\[
I^{\cX}_{\bv{k}}(t,z):= z\sum_{k_0\geq 0}\frac{t^{k_0+d\lambda/z}z^{k_0 r/d}}{z^{\sum_j\langle k_0q_j+a(\bv{k})^j\rangle}k_0!} \prod_{j=1}^N \frac{\Gamma(1-c_j\tfrac{\lambda}{z}- \langle k_0q_j+a(\bv{k})^j\rangle)}{\Gamma(1-c_j\tfrac{\lambda}{z}- k_0q_j-a(\bv{k})^j)}\ii_{\jj^{k_0}\prod_g g^{k_g}}.
\]
Then the $I$-functions of $\cY$ and $\cX$ can be respectively written as
\[
\begin{split}
I^{\cY}(q,\bt,z)&= \sum_{\bv{k}\in(\ZZ_{\geq 0})^{\bar{G}}}\prod_{g \in \bar{G}}\frac{(t^{g})^{k_g}z^{(\age(g)-1)k_g}}{k_g!} I^{\cY}_{\bv{k}}(q,z)\\
I^{\cX}(t,\bt,z)&= \sum_{\bv{k}\in(\ZZ_{\geq 0})^{\bar{G}}}\prod_{g \in \bar{G}}\frac{(t^{g})^{k_g}z^{(\age(g)-1)k_g}}{k_g!} I^{\cX}_{\bv{k}}(t,z).
\end{split}
\]
\end{notation}

\subsection{Setting $z$ equal to $1$}\label{z=1}
In what follows we will set $z=1$ in the $I$-functions for computational convenience. It is worth noting that we can recover the original $I$-functions by the following procedure. Define a grading operator $\gr$ by
\begin{equation}
\begin{split}
&\gr (H):=H,\quad \gr (\lambda):=\lambda,\quad\gr(t^{g}):=(1-\age(g))t^{g},\\
& \gr\left(\tilde \ii_{\jj^{-k_0}\prod_g g^{k_g}}\right):= \left(\sum_j\langle k_0q_j-a(\bv k)^j\rangle\right)\tilde \ii_{\jj^{-k_0}\prod_g g^{k_g}},\quad\text{and}\\
& \gr\left(\ii_{\jj^{k_0}\prod_g g^{k_g}}\right):=\left(\sum_j\langle k_0q_j+a(\bv{k})^j\rangle \right) \ii_{\jj^{k_0}\prod_g g^{k_g}}.
 \end{split}
\end{equation}
We then have the following relations that allow us to restore the powers of $z$ to the $I$-functions:
\[
\begin{split}
I^{\cY}(q,\bt,z)&= z^{1-\gr}z^{rH}I^{\cY}(q/z^r,\bt,1)\quad\text{and}\\
I^{\cX}(t, \bt,z)&=z^{1-\gr}z^{-r\lambda}I^{\cX}(t z^{r/d}, \bt,1).
\end{split}
\]
\subsection{The regularized I-function}
We define the \textit{regularized $I$-function} of $[\CC^N/G]$ to be
\[
I^{reg}_{\bv{k}}(\tau):=\sum_{k_0\geq 0} \frac{\tau^{r(k_0/d+\lambda)}}{\Gamma(1+r\left(k_0/d+\lambda \right))k_0!}\prod_{j=1}^N \frac{\Gamma(1-c_j\lambda- \langle k_0q_j+a(\bv{k})^j\rangle)}{\Gamma(1-c_j \lambda- k_0q_j-a(\bv{k})^j))}  \ii_{\jj^{k_0}\prod_g g^{k_g}}.
\]
Using the ratio test it is easy to see that this series defines a holomorphic function in a disk of radius $\left(r^r d^d \prod_{j=1}^N c_j^{-c_j}\right)^{1/r}$ centered at $\tau=0$. A simple computation reveals that $I^{reg}_{\bv{k}}(\tau)$ satisfies the following \textit{regularized Picard-Fuchs equation}:
\begin{multline}\label{regularized_equation}
\left[\tau^r\prod_{j=1}^N\prod_{l=0}^{c_j-1}\left(-\frac{c_j}{r}\tau\frac{d}{d\tau}-l-a(\bv{k})^j\right)\right.\\
\left.-\prod_{l=0}^{d-1}\left(\frac{d}{r}\tau\frac{d}{d\tau}-d\lambda-l\right)\prod_{l=0}^{r-1}\left(\tau\frac{d}{d\tau}-l\right)\right]I_{\bv{k}}^{reg}(\tau)=0.
\end{multline}
This equation has singular points at $\tau=0,\infty$ and for $\tau$ satisfying $(-\tfrac{\tau}{r})^r=(-d)^d\prod_{j=1}^Nc_j^{-c_j}$. All of these points are regular singularities. This means that $I^{reg}_{\bv{k}}(\tau)$ can be analytically continued to $\tau=\infty$ along any ray that avoids these singularities.
\subsection{Proof of Theorem \ref{theorem_asymptotic_correspondence}} The idea behind the proof of theorem \ref{theorem_asymptotic_correspondence} consists of constructing a holomorphic function that satisfies the Picard-Fuchs equation of $\cY$, and whose asymptotic expansion is given by the $I$-function of $\cX$.
To construct this function we use the following Laplace integral of the regularized $I$-function:
\[
\iI_{\bv{k}}(u):=u\int_{0}^{\infty}e^{-u\tau}I_{\bv{k}}^{reg}(\tau)d\tau,
\]
where the ray of integration avoids the singular points of $I_{\bv{k}}^{reg}(\tau)$. This integral is well-defined, as shown in \cite[Chapter 2]{MKH}.
As a consequence of Watson's lemma for a complex variable \cite{Miller}, the asymptotic expansion of $\iI_{\bv{k}}(u)$  is given by
\[
\iI_{\bv{k}}(u)\sim \sum_{k_0\geq 0} \frac{1}{u^{r(k_0/d+\lambda)}k_0!}\prod_{j=1}^N \frac{\Gamma(1-c_j\lambda- \langle k_0q_j+a(\bv{k})^j\rangle)}{\Gamma(1-c_j \lambda- k_0q_j-a(\bv{k})^j))}  \ii_{\jj^{k_0}\prod_g g^{k_g}}
\]
as $u\rightarrow\infty$ from the region $|\arg(u)|<\text{min}(\pi/r,\pi/2)$ if $\sum_j c_j$ is odd and from the region $0<\arg(u)<\text{min}(2\pi/r,\pi/2)$ if $\sum_{j} c_j$ is even.
It follows from this that 
\begin{equation}
\iI_{\bv{k}}(u=q^{1/r})\sim I^{\cX}_{\bv{k}}(t=q^{-1/d},z=1)
\end{equation}
as $q\rightarrow\infty$.

Define $\iI (u)$ to be
\begin{equation}\label{holomorphic_I_function}
\iI (u):=\sum_{\bv{k}\in(\ZZ_{\geq 0})^{\bar{G}}}\prod_{g \in \bar{G}}\frac{(t^{g})^{k_g}}{k_g!} \iI_{\bv{k}}(u).
\end{equation}

The following lemma and corollaries show that $\iI$ satisfies the Picard-Fuchs equation of $\cY$.
\begin{lemma}\label{lemma_picard_fuchs}
$\iI_{\bv{k}}(u)$ satisfies the following differential equation:
\begin{equation}\label{picard1}
\left[\prod_{j=1}^N\prod_{l=0}^{c_j-1}\left(\frac{c_j}{r}u\frac{d}{du}-l-a(\bv{k})^j\right)-u^r\prod_{l=0}^{d-1}\left(-\frac{d}{r}u\frac{d}{du}-d\lambda-l\right) \right]\iI_{\bv{k}}(u)=0.
\end{equation}
Furthermore, given $\bv{k} \in (\ZZ_{\geq 0})^{\bar{G}}$, let
$\bv{g} = \prod_{g \in \bar{G}} g^{k_g}$.  Then for all $1 \leq j \leq N$, 
 $a(\bv{k})^j- m_j(\bv{g})  = M_j$ for some integer $M_j \geq 0$ and 
\begin{equation}\label{picard2}
\prod_{j=1}^N\prod_{l=0}^{M_j-1}\left(\frac{c_j}{r}u\frac{d}{du} -l-m_j(\bv{g})\right)\iI_{\bv{k(g)}}(u)=\iI_{\bv{k}}(u),
\end{equation}
where 
\begin{equation}\label{e:kgb}
\bv{k(g)}:  g \mapsto \left\{ \begin{array}{ll}
1 & \text{if $g = \bv{g}$}\\
0 & \text{otherwise.}
\end{array}
\right.
\end{equation}
\end{lemma}
\begin{proof}
Let $f(\tau)$ be holomorphic in some region of the complex plane containing a ray on which we can define the Laplace transform of $f$. Moreover, let $f(0)=0$ and define $F(u):=u\int_{0}^{\infty} e^{-u\tau}f(\tau)d\tau$. Using the properties of the Laplace transform it is easy to see that 
\begin{align} \nonumber
u\mathcal{L}\left(\alpha \tau\frac{d}{d\tau}f(\tau)+\beta f(\tau)\right)&=-\alpha u \frac{d}{du} \left(\mathcal{L}\left(\frac{d}{d\tau}f(\tau)\right)\right)+\beta u\mathcal{L}\left(f(\tau)\right)\\ \nonumber
&=-\alpha u \frac{d}{du} \left(u\mathcal{L}(f(\tau))-f(0)\right)+\beta u\mathcal{L}\left(f(\tau)\right)\\
&=\left(-\alpha u \frac{d}{du} +\beta\right)F(u), \label{Laplace_property}
\end{align}
where $\alpha$ and $\beta$ are arbitrary complex numbers.

Now, note that Equation \ref{regularized_equation} can be rewritten as
\[
\left[\prod_{j=1}^N\prod_{l=0}^{c_j-1}\left(-\frac{c_j}{r}\tau\frac{d}{d\tau}-l-a(\bv{k})^j\right)-\left(\frac{d}{d\tau}\right)^r\prod_{l=0}^{d-1}\left(\frac{d}{r}\tau\frac{d}{d\tau}-d\lambda-l\right)\right]I_{\bv{k}}^{reg}(\tau)=0.
\]
Applying Equation \ref{Laplace_property} iteratively to this yields Equation \ref{picard1}.

To establish Equation \ref{picard2}, note that $I^{reg}_{\bv{k}}(\tau)$ satisfies 
\[
\prod_{j=1}^N\prod_{l=0}^{M_j-1}\left(-\frac{c_j}{r}\tau\frac{d}{d\tau} -l-m_j(\bv{g})\right) I_{\bv{k(g) }}^{reg}(\tau)= I_{\bv{k}}^{reg}(\tau)
\]
Applying Equation \ref{Laplace_property} iteratively to this yields the desired result.
\end{proof}
\begin{corollary}\label{PF-corollary}
$\iI_{\bv{k}}(u=q^{1/r})$ satisfies 
\begin{equation}\label{PF_equation_k}
\left[\prod_{j=1}^N\prod_{l=0}^{c_j-1}\left(c_j q\frac{d}{dq}-l-a(\bv{k})^j\right)-q\prod_{l=0}^{d-1}\left(-dq\frac{d}{dq}-d\lambda-l\right) \right]\iI_{\bv{k}}(u=q^{1/r})=0
\end{equation}
for all $g \in \bar{G}$.  Furthermore, given $\bv{k} \in (\ZZ_{\geq 0})^{\bar{G}}$, then with the same notation as in the previous lemma, we have
\begin{equation}\label{e:hopeitsnotwrong}
\prod_{j=1}^N\prod_{l=0}^{M_j-1}\left(c_j q\frac{d}{dq} -l-m_j(\bv{g})\right)\iI_{\bv{k(g)}}(u=q^{1/r})=\iI_{\bv{k}}(u=q^{1/r}).
\end{equation}

\end{corollary}
\begin{proof}
This is a consequence of Lemma~\ref{lemma_picard_fuchs} and the change of variables $q=u^r$.
\end{proof}
\begin{remark} Note that, still using the notation of the previous lemma, one can also check that
\begin{equation}\label{e:remarkeq}
\prod_{j=1}^N\prod_{l=0}^{M_j-1}\left(c_j q\frac{d}{dq} -l-m_j(\bv{g})\right)I^{\cY}_{\bv{k(g)}}(q,z=1)=I^{\cY}_{\bv{k}}(q,z=1).
\end{equation}
\end{remark}
\begin{corollary}\label{corollary_linear_k} For all $\bv{g} \in \bar{G}$,
there exists a unique linear transformation \[L_{\bv{g}}:H^*_{\bv{g}}(\cY) \rightarrow H^*_{\bv{g}}(\cX)\] such that for all $ \bv{k}\in(\ZZ_{\geq 0})^{\bar{G}}$ satisfying $  \prod_{g \in \bar{G}} g^{k_g} = \bv{g}$,
\begin{equation}\label{linear_k}
L_{\bv{g}}\cdot I^{\cY}_{\bv{k}}(q,z=1)=\iI_{\bv{k}}(u=q^{1/r}).
\end{equation}

\end{corollary}
\begin{proof}Since the components of $I^{\cY}_{\bv{k}}(q,z=1)$ are a complete set of solutions to the Picard-Fuchs Equation \ref{PF_equation_k}, by Corollary \ref{PF-corollary} there exists a unique linear transformation $L_{\bv{g}}:H^*_{\bv{g}}(\cY) \rightarrow H^*_{\bv{g}}(\cX)$ such that
\[
L_{\bv{g}}\cdot I^{\cY}_{\bv{k(g)}}(q,z=1)=\iI_{\bv{k(g)}}(u=q^{1/r}),
\]
where $\bv{k(g)}\in(\ZZ_{\geq 0})^{\bar{G}}$ is defined as in~\eqref{e:kgb}.
Applying the differential operator 
\[
\prod_{j=1}^N\prod_{l=0}^{M_j-1}\left(c_j q\frac{d}{dq} -l-m_j(\bv{g})\right)
\]
from Corollary~\ref{corollary_linear_k} to the equation above,
we conclude that \eqref{linear_k} holds whenever $\prod_{g \in \bar{G}} g^{k_g } =\bv{g}$. 
\end{proof}


\textit{Proof of Theorem \ref{theorem_asymptotic_correspondence}}:\\
%
By Corollary \ref{corollary_linear_k} we see that for all $\bv{g} \in \bar{G}$, there exists a unique linear transformation $L_{\bv{g}} :H_{\bv{g}}^*(\cY) \longrightarrow H_{\bv{g}}^*(\cX)$ (recall notations~\ref{n:Gcosets1} and~\ref{n:Gcosets2}) such that $L_{\bv{g}} \cdot I^{\cY}_{\bv{k}}(q,\bt,z=1)=\iI_{\bv{k}} (u=q^{1/r})$ whenever $\prod_{g \in \bar{G}} g^{k_g} = \bv{g}$.
We define $L$ to be the block-diagonal sum \[L = \bigoplus_{\bv{g} \in \bar{G}} L_{\bv{g}}: H^*_{CR, T}(\cY) = \bigoplus_{\bv{g} \in \bar{G}} H^*_{\bv{g}}(\cY) \longrightarrow H^*_{CR,T}(\cX) = \bigoplus_{\bv{g} \in \bar{G}} H^*_{\bv{g}}(\cX).\]
It immediately follows that $L\cdot I^{\cY}(q,\bt,z=1)=\iI (u=q^{1/r})$. We now compute its asymptotic expansion:
\[
\begin{split}
L\cdot I^{\cY}(q,\bt,z=1)&=\iI (u=q^{1/r})\\
&=\sum_{\bv{k}\in(\ZZ_{\geq 0})^{\bar{G}}}\prod_{g \in \bar{G}}\frac{(t^{g})^{k_g}}{k_g!} \iI_{\bv{k}}(u=q^{1/r})\\
&\sim \sum_{\bv{k}\in(\ZZ_{\geq 0})^{\bar{G}}}\prod_{g \in \bar{G}}\frac{(t^{g})^{k_g}}{k_g!} I^{\cX}_{\bv{k}}(t=q^{-1/d},z=1)\quad\text{as $q\rightarrow \infty$}\\
&=I^{\cX}(t=q^{-1/d}, \bt,z=1).
\end{split}
\]
This concludes the proof of the theorem.
\begin{flushright}
\qedsymbol
\end{flushright}
\subsection{The correspondence for $\text{disc}(|f|)<0$}\label{correspondence_r_negative} For the case in which $\text{disc}(|f|)<0$, we have a result analogous to Theorem~\ref{theorem_asymptotic_correspondence}. In this case, however, the roles of $\cY$ and $\cX$ are interchanged. This means that it is possible to obtain the genus zero Gromov-Witten invariants of $\cY$ from the genus zero invariants of $\cX$. 
\begin{theorem}\label{theorem_asymptotic_correspondence2}
Let $\text{disc}(|f|)<0$. There exists a  unique linear transformation $L:H_{CR,T}^{\ast}(\cX)\longrightarrow H_{CR,T}^{\ast}(\cY)$ such that 
\[
L\cdot I^{\cX}(t,\bt,z=1)\sim I^{\cY}(q=t^{-d}, \bt,z=1)\quad\text{as $t\rightarrow\infty$}.
\]
\end{theorem}
The proof of this result is almost identical to the proof of Theorem \ref{theorem_asymptotic_correspondence} once we interchange the roles of $\cX$ and $\cY$.

As in the case $\text{disc}(|f|)>0$, we have the following important consequence:
\begin{corollary}
If  $\text{disc}(|f|)<0$ then the genus zero Gromov-Witten invariants of $\cY$ are completely determined by the genus zero Gromov-Witten invariants of $\cX$.
\end{corollary}
\section{FJRW theory and Landau--Ginzburg correspondences}\label{s:FJRW}

In this section we prove that the FJRW theory of the pair $(W,G)$ from Section~\ref{s:BU} is related to the Gromov--Witten theory of the hypersurface $\cZ:=\{W = 0\} \subset \PP(G)$ via asymptotic expansion as in the previous section.  This generalizes the results of the first author in \cite{A} beyond the case $\PP(G) = \PP(c_1, \ldots, c_N)$, and may be viewed as extending the Landau--Ginzburg/Calabi--Yau correspondence \cite{ChR, ChIR, PS, LPS} to the setting where $\cZ$ is no longer Calabi--Yau.  For simplicity we will term these types of theorems \emph{Landau--Ginzburg correspondences}.
While not strictly a consequence of Theorems~\ref{theorem_asymptotic_correspondence} and~\ref{theorem_asymptotic_correspondence2}, we will show how the proofs above may be modified slightly to yield the desired correspondences.

\subsection{Setup}
We first recall notation from Section~\ref{s:BU}.  We are given a Fermat polynomial $W = X_1^{d/c_1} + \cdots + X_N^{d/c_N}$ with $\gcd(c_1, \ldots, c_N) = 1$ and a group $G$ of diagonal automorphisms of $W$.  $W$ is quasihomogeneous of degree $d$:
\[W(\alpha^{c_1}X_1, \ldots, \alpha^{c_N}X_N) = \alpha^d W(X_1, \ldots, X_N)\]
for all $\alpha \in \CC$.
We assume that $G$ contains the distinguished element $\jj = \exp(2\pi i \diag (q_1, \ldots, q_N))$,
where $q_j = c_j/N$.  For the correspondence to hold, we also require  that the line bundle $\sO(-d) \to \PP(G)$ be pulled back from the coarse underlying space of $\PP(G)$.  This corresponds to the somewhat technical condition that all $g \in G$ which fix at least one coordinate lie in $SL_N(\CC)$.  We call $(W,G)$ a \emph{Landau--Ginzburg} pair.
\subsection{FJRW theory}

Given a non-degenerate quasi-homogeneous polynomial $W = W(X_1, \ldots, X_N)$ and an \emph{admissible} group $G$ \cite{FJR1}, one can define the corresponding \emph{FJRW} invariants, which may be viewed as the analogue of Gromov--Witten invariants for the singular space $\{W = 0\} \subset [\CC^N/G]$.  They are defined as integrals over a cover of the moduli space of stable curves, and obey many of the same axioms as Gromov--Witten invariants.  

In general the definition of FJRW invariants is quite involved, but in genus zero, when $W$ is a Fermat polynomial, the situation simplifies greatly.  In what follows we will always assume that the conditions on $(W,G)$ given above are satisfied.

We first define the state space.
\begin{definition}  \label{d:2.5}
Given a Landau--Ginzburg pair $(W,G)$, the \emph{narrow FJRW state space} is given by 
\[
 \cH_{FJRW}(W, G) := \bigoplus_{g \in G_{nar}} \CC \phi_g,
\] 
where
\[G_{nar} : =\{g \in G| g\jj \text{ fixes only the origin in } \CC^N\}\]
and $\phi_g$ is a vector formally associated to $g \in G_{nar}$.
\end{definition}

%
\begin{definition}
Given a Landau--Ginzburg pair $(W,G)$ as above, the  moduli space $W_{h,n; G}$ of $W$-structures parametrizes 
families of orbifold curves $\cC$ together with an $N$-tuple $(\cL_j, \varphi_j)$ such that
\begin{itemize}
\item $\cC$ has $\mu_d$ isotropy at each marked point and node, and the coarse space $|\cC|$ is a (family of) $n$-pointed genus $h$ stable curve,
\item $\cL_j$ is a line bundle on $\cC$ and $\varphi_j$ is an isomorphism $\varphi_j: \cL_j^{\otimes d} \stackrel{\sim}{\to} \omega_{\cC, \log}^{\otimes c_j}$ such that,
\item for any Laurent monomial  $\prod_{j = 1}^N X_j^{b_j}$ invariant under $G$, 
\[\bigotimes_{j = 1}^N \cL_j^{\otimes b_j} \cong \omega_{\cC, \log}.\]
\end{itemize}
\end{definition}
See \cite{FJR1, LPS} for more details.
%
%
%

As a consequence of the third condition above, given a point $(\cC, \cL_j, \phi_j)$ in $W_{h, n; G}$, at each marked point $p_i$ the isotropy acts on fibers of $\oplus_{j=1}^N \cL_j$ by an element of $G$.  
One may therefore decompose $W_{h, n, G}$ into a union of open and closed substacks based on the action of the corresponding isotropy group.  Let 
\[  W_{h, n, G}(g_1, \ldots, g_n)  \] 
denote the substack where the isotropy generator at $p_i$ acts by $g_i$.  

In analogy to Gromov--Witten theory, FJRW invariants are defined by integrating against a virtual cycle on $W_{h,n; G}$.
When $W$ is a Fermat polynomial and the genus is zero,
 one can prove \cite{FJR1} that 
$$R^0\pi_* (\oplus_{j=1}^N \cL_j) = 0$$ and 
$$-R\pi_* (\oplus_{j=1}^N \cL_j) =  R^1(\oplus_{j=1}^N \cL_j)[-1] $$ 
is a vector bundle.  Let $\phi_{g_1}, \ldots, \phi_{g_n}$ be elements of the state space $\cH_{FJRW}(W,G)$, and let $a_1, \ldots, a_n$ be non-negative integers.
By axiom (5a) of \cite[Theorem~4.1.8]{FJR1}, one defines
\begin{equation} \label{e:2.2.1}
\br{ \psi^{a_1} \phi_{ g_1}, \ldots, \psi^{a_n} \phi_{  g_n} }^{(W,G)}_{0,n} 
:=  {\bar d}^N \int_{W_{0, n, G}(g_1\jj, \ldots, g_n\jj)}
\frac{\prod_{i = 1}^n \psi_i^{a_i}}{e \left(R\pi_* ( \oplus_{{j}=1}^N   \cL_{j})^{\vee} \right)}.
\end{equation}  

FJRW theory is a cohomological field theory.  Furthermore, all of the axioms of an \emph{axiomatic Gromov--Witten theory} \cite{YPnotes, LPS} are satisfied.  Therefore quantum cohomology, the Dubrovin connection, the $J$-function, and $I$-functions may all be defined and behave in the same manner as in Gromov--Witten theory.  Thus Section~\ref{s:formal} transfers almost word for word to the FJRW setting.  See \cite{LPS}, Section~3, for more details.
We will denote the $J$-function and $I$-functions coming from the FJRW theory of $(W, G)$ by $J^{(W,G)}$ and $I^{(W,G)}$ respectively.

\begin{remark}
The reader may observe a strange shift in Definition~\eqref{e:2.2.1},  the $\phi_{g_i}$ insertion at the $i$th marked point corresponds to an isotropy action of $g_i \jj$.  This shift is chosen to so that the $\phi_{id}$ element of the state space is the identity in quantum cohomology.  
\end{remark}

\begin{remark}
FJRW theory is defined for a more general class of insertions corresponding to $g \in G \setminus G_{nar}$ (in analogy with primitive cohomology classes of a hypersurface).  We will content ourselves in this paper with statements relating the narrow invariants from $G_{nar}$ to the ambient cohomology of $\cZ$.
\end{remark}

\subsection{$I$-functions}

Rather than directly compute the relevant $I$-functions for the Landau--Ginzburg correspondences, we will apply two results, the \emph{multiple log canonical} (MLK) and \emph{quantum Serre duality} (QSD) correspondences to obtain $I^{(W,G)}$ and $I^{\cZ}$ from $I^{\cX}$ and $I^{\cY}$ respectively. The MLK correspondence says roughly that $I^{(W,G)}$ is related to $I^{\cX}$ by differentiation, linear transformation, and a nonequivariant limit.  As a consequence, we will see that the Picard--Fuchs
differential equations satisfied by $I^{(W,G)}$ are closely related to those for $I^{\cX}$ described in the previous section.  The QSD correspondence is analogous.

\subsubsection{The $I$-function of $\cX$}

Applying the MLK correspondence \cite[Theorem~5.12]{LPS}, an FJRW $I$-function for $(W, G)$ is given by  
\[I^{(W,G)}(t, \bt,z) =  \lim_{\lambda \mapsto 0}\Delta^\circ \left( z t \frac{d}{dt} I^\cX(t, \bt, z)\right),\]
where $\Delta^\circ$ is the linear map defined by
\[\Delta^\circ:= \ii_g \mapsto
\left\{ \begin{array}{ll}
(-1)^{\age(g)} \phi_{g \jj^{-1}} & \text{if $g$ fixes only the origin}\\
0 & \text{otherwise.}
\end{array}
\right.
\]
%
We obtain:
\begin{align}
 \nonumber
I^{(W,G)}(t, \bt,z) = &z\sum_{\bv{k}\in(\ZZ_{\geq 0})^{\bar{G}}}\prod_{g \in \bar{G}}\frac{(t^{g})^{k_g}z^{(\age(g)-1)k_g}}{k_g!}\sum_{k_0\geq 0}\frac{t^{k_0+1}z^{1 + (k_0+1)r/d}}{z^{\sum_j\langle (k_0+1)q_j+a(\bv{k})^j\rangle}k_0!}\\
  & \cdot \prod_{j=1}^N \frac{\Gamma(1- \langle (k_0+1)q_j+a(\bv{k})^j\rangle)}{\Gamma(1- (k_0+1)q_j-a(\bv{k})^j)}(-1)^{\age(\jj^{k_0}\prod_g g^{k_g})}\phi_{\jj^{k_0}\prod_g g^{k_g}} \label{e:IWG}
\end{align}
where $\phi_g$ is understood to be  zero if $g \notin G_{nar}$.

Now, given $\bv{g} \in \bar{G}$, consider the $t^{\bv{g}}$-coefficient of $I^{(W,G)}(t, \bt,z)$, denote it by $I^{(W,G)}_{\bv{g}}(t, \bt,z)$:
\begin{align}
\nonumber
I^{(W,G)}_{\bv{g}}(t, \bt,z) = &z\sum_{k_0\geq 0}\frac{t^{k_0+1}z^{1 + (k_0+1)r/d}}{z^{\sum_j\langle (k_0+1)q_j+m_j(\bv{g})\rangle}k_0!}\\
  & \cdot \prod_{j=1}^N \frac{\Gamma(1- \langle (k_0+1)q_j+m_j(\bv{g})\rangle)}{\Gamma(1- (k_0+1)q_j-m_j(\bv{g}))}(-1)^{\age(\jj^{k_0}\prod_g g^{k_g})} \phi_{\jj^{k_0}\bv{g}}  \label{e:IWGg}.
\end{align}  

Note that since $\phi_{j^k \bv{g}} = 0$ if $j^k \bv{g}$ is not narrow, the function $I^{(W,G)}_{\bv{g}}(t, \bt,z)$ is supported on the space 
\[\cH_{\bv{g}}(W,G) := \bigoplus_{\left\{ \substack{ 0 \leq k \leq d-1 \\ j^k\bv{g} \in G_{nar}} \right\} } \CC \phi_{j^k g}
\] of dimension $\dim(H^*_{\bv{g}}(\cX)) - \#\{0 \leq k \leq d-1| j^k \bv{g} \text{ fixes a coordinate}\}$.

We claim that $I^{(W,G)}_{\bv{g}}(t, \bt,z)$ satisfies a Picard--Fuchs equation obtained from that of 
$I^\cX_{\bv{g}}(t, \bt, z)$ by removing factors.  The nonequivariant limit of the original equation for $\cX$ is $\cD_{t, \bv{g}} = 0$, where
\begin{align*} \cD_{t, \bv{g}} =& \prod_{l=0}^{d-1} \left( zt\frac{d}{dt} - lz\right) - t^d \prod_{j=1}^N \prod_{l = 0}^{c_j - 1} \left((-c_j/d) zt\frac{d}{dt} - lz - m_j(\bv{g}) z\right)
\end{align*} 
Note first that by construction of our splitting $\bar{G}$, $m_1(\bv{g}) = 0$.  Thus we can factor a $zt\frac{d}{dt}$ from the right side of $\cD_{t, \bv{g}}$ since the $j = 1$, $l = 0$ term in the right hand product is a multiple of $zt\frac{d}{dt}$.  It is immediate that $I^{(W,G)}_{\bv{g}}(t, \bt,z)$ is annihilated by this reduced operator, because it was obtained by applying $zt\frac{d}{dt}$ to $I^\cZ_{\bv{g}}(t, \bt, z)$.  Next, assume that $\jj^k \bv{g}$ fixes coordinate $i$, for some $1 \leq k \leq d-1$.   We claim that $\cD_{t, \bv{g}}$ contains a factor of $(zt\frac{d}{dt} - kz)$ on the left.

Note first that $m_i(\bv{g}) = f c_i/d$ for some $1 \leq f \leq d/c_i - 1$, thus $k = m d/c_i - f$ for $1 \leq m \leq c_i$.  Consider the factor $ \left((-c_j/d )zt\frac{d}{dt} - lz - m_j(\bv{g}) z\right)$ in the right hand product of $\cD_{t, \bv{g}}$ for $l = c_i - m$.  Commuting this with $t^d$ yields
\begin{align*}
t^d \left((-c_i/d) zt\frac{d}{dt} - (c_i - m)z - m_i(\bv{g}) z\right) = &\left((-c_i/d) (zt\frac{d}{dt} - d)- (c_i - m) z - m_i(\bv{g}) z\right)t^d \\
=& -(c_i/d) \left( zt\frac{d}{dt} - (m d/c_i -  m_i(\bv{g})d/c_i )\right) t^d\\
=& -(c_i/d) \left( zt\frac{d}{dt} - k \right) t^d.
\end{align*}
We can thus factor $\left( zt\frac{d}{dt} - k \right)$ from the left side of $\cD_{t, \bv{g}}$ for every power of $k$ such that $\jj^k \bv{g}$ fixes a coordinate.  We obtain
\begin{equation}\label{e:PFred}
\cD_{t, \bv{g}} = \prod_{\left\{ \substack{ 1 \leq k \leq d-1 \\ j^k\bv{g} \text{ fixes a coordinate}} \right\}}\left( zt\frac{d}{dt} - k \right)\cdot  \cD_{t, \bv{g}}^{irr} \cdot zt\frac{d}{dt}.
\end{equation}
One can easily check that the $I$-function \eqref{e:IWGg} satisfies the Picard--Fuchs equation $\cD_{t, \bv{g}}^{irr} = 0$.

\subsubsection{The $I$-function of $\cZ$}
Let $i: \cZ \to \PP(G)$ be the inclusion of the hypersurface defined by the equation $\{W = 0\}$.  
By Quantum--Serre--Duality of Coates--Givental \cite{CG} and Tseng \cite{Ts}, an $I$-function for $\cZ$ is given by
\[I^{\cZ}(q,\bt,z)  = \lim_{\lambda \mapsto 0} i^* \left( \frac{ z  q\frac{d}{dq} I^\cY(\pm q, \bt, z)}{-d(\lambda + H)}\right).
\]
where $\pm q$ denotes the substitution $q^{k_0/d} \mapsto (-1)^{k_0} q^{k_0/d}$.
We obtain:
\begin{align} \nonumber
I^{\cZ}(q,\bt,z) =&zq^{H/z} \sum_{\bv{k}\in(\ZZ_{\geq 0})^{\bar{G}}}\prod_{g \in \bar{G}}\frac{(t^{g})^{k_g}z^{(\age(g)-1)k_g}}{k_g!} \sum_{k_0\geq 0}\frac{ (-1)^{k_0}q^{k_0/d}}{z^{k_0r/d + \sum_j\langle k_0q_j-a(\bv k)^j\rangle }} \\ 
	&\cdot \frac{\Gamma(-dH/z)}{\Gamma(-k_0-dH/z)}\prod_{j=1}^N\frac{\Gamma(1+c_jH/z-\langle -k_0q_j+a(\bv{k})^j \rangle)}{\Gamma(1+c_jH/z+k_0q_j-a(\bv{k})^j)}\tilde \ii_{\jj^{-k_0}\prod_g g^{k_g}},
\end{align}
where we now consider cohomology classes $H^k\tilde \ii_{\jj^{-k_0}\prod_g g^{k_g}}$ as elements of $H^*_{CR}(\cZ)$ via the pullback to the hypersurface.  $I^{\cZ}(q,\bt,z)$ gives a big $I$-function for the \emph{ambient part} of the Gromov--Witten theory of $\cZ$.  In other words, $I^{\cZ}(q,\bt,z)$ determines the quantum cohomology ring of $\cZ$ restricted to classes in $H^{amb}_{CR}(\cZ) \subset H^*_{CR}(\cZ)$ defined as the image of $i^*: H^*_{CR}(\PP(G)) \to H^*_{CR}(\cZ)$.

As in the previous case, given $\bv{g} \in \bar{G}$, let $I^{\cZ}_{\bv{g}}(t, \bt,z)$ the part of $I^{\cZ}(t, \bt,z)$ coming from the $t^{\bv{g}}$-coefficient:
\begin{align} \nonumber
I^{\cZ}_{\bv{g}}(q,\bt,z) =&zq^{H/z}  \sum_{k_0\geq 0}\frac{ (-1)^{k_0}q^{k_0/d}}{z^{k_0r/d + \sum_j\langle k_0q_j-m_j(\bv{g})\rangle}} \\ 
	&\cdot \frac{\Gamma(-dH/z)}{\Gamma(-k_0-dH/z)}\prod_{j=1}^N\frac{\Gamma(1+c_jH/z-\langle -k_0q_j+a(\bv{k})^j \rangle)}{\Gamma(1+c_jH/z+k_0q_j-m_j(\bv{g}))}\tilde \ii_{\jj^{-k_0}\prod_g g^{k_g}}.
\end{align}

Note that top powers of $H$ on any connected component of $I\PP(G)$ will vanish under the pullback.  Thus $I^{\cZ}_{\bv{g}}(t, \bt,z)$ is supported on the space 
\[H^*_{\bv{g}}(\cZ) := i^* \left( H^* \left( \cup_{0 \leq k \leq d-1} \PP(G)_{\jj^k \bv{g}}\right) \right).\]
 of dimension $\dim(H^*_\bv{g}(\cY)) - \#\{0 \leq k \leq d-1| j^k \bv{g} \text{ fixes a coordinate}\}$.


\subsection{Landau--Ginzburg correspondences} We begin this section by establishing the Landau-Ginzburg/Fano correspondence. We assume that $\sum c_j > d$, which implies that $\cZ$ is a Fano hypersurface. The main result we prove in this section is that the genus zero Gromov-Witten theory of $\cZ$ completely determines the genus zero FJRW theory of the pair $(W,G)$. The technical details of the proof of this result are almost identical to the details found Section \ref{the correspondence for toric blowups}. For this reason we avoid some computations.

We start by noting that $I^{\cZ}_{\bv{g}}(q,z)$ satisfies the equation $\cD_{q,\bv{g}}^{irr} = 0$, where the differential operator $\cD_{q,\bv{g}}^{irr}$ is 
equal to $\cD_{t, \bv{g}}^{irr}$  of~\eqref{e:PFred}, but with the change of variables $q = (-t)^d$.


Since the degree of the differential operator $\cD_{q,\bv{g}}^{irr}$ is the same as the dimension of $H^*_{\bv{g}}(\cZ)$, it follows that the components of $I^{\cZ}_{\bv{g}}(q,z)$ are a complete set of solutions to the equation $\cD_{q,\bv{g}}^{irr} = 0$. We next construct a holomorphic function whose asymptotic expansion at infinity is given by $I^{(W,G)}_{\bv{g}}$. Define the regularized version of $I^{(W,G)}_{\bv{g}}$ to be

\begin{align}\nonumber
I^{(W,G),reg}_{\bv{g}}(\tau): = &\sum_{k_0\geq 0}\frac{(-1)^{k_0+1}\tau^{r(k_0+1)/d}}{\Gamma(1+\frac{r}{d}(k_0+1))k_0!}\\
  & \cdot \prod_{j=1}^N \frac{\Gamma(1- \langle (k_0+1)q_j+m_j(\bv{g})\rangle)}{\Gamma(1- (k_0+1)q_j-m_j(\bv{g}))}(-1)^{\age(\jj^{k_0}\prod_g g^{k_g})} \phi_{\jj^{k_0}\bv{g}} .
\end{align}
This series converges on a disk of finite radius centered at $\tau=0$. Moreover, it satisfies a regularized Picard-Fuchs equation $\cD_{\tau,\bv{g}}^{irr, reg}=0$ with regular singularities at $\tau=0,\infty$ and for $\tau$ satisfying $(-\tfrac{\tau}{r})^r=d^d\prod_{j=1}^Nc_j^{-c_j}$. Thus, $I^{(W,G),reg}_{\bv{g}}(\tau)$ can be analytically continued to $\tau=\infty$ along any ray that avoids these singularities.
Define
\[
\mathbb{I}^{(W,G)}_{\bv{g}}(u):=u\int_{0}^{\infty}e^{-u\tau}I^{(W,G),reg}_{\bv{g}}(\tau)d\tau,
\]
where the ray on integration avoids the singular points of $\cD_{\tau,\bv{g}}^{irr, reg}$.
By Watson's lemma, we have that 
\begin{equation}\label{Fano asymptotic equation}
\mathbb{I}^{(W,G)}_{\bv{g}}(u=q^{1/r})\sim I^{(W,G)}_{\bv{g}}(t=-q^{-1/d}, z=1)\text{ as $q\rightarrow\infty$.} 
\end{equation}
After a computation similar to the found in the proof of Corollary \ref{PF-corollary}, one sees that  $ \cD_{q,\bv{g}}^{irr}\cdot\mathbb{I}^{(W,G)}_{\bv{g}}(u=q^{1/r})=0$. We now prove a result similar to Corollary \ref{corollary_linear_k}:
\begin{lemma}\label{Lg for LG/Fano}
For each $\bv{g} \in \bar{G}$, there exists a unique linear transformation $L_{\bv{g}}:H^*_{\bv{g}}(\cZ) \rightarrow \cH_{FJRW}(W, G)$ such that for all $\bv{k} \in(\ZZ_{\geq 0})^{\bar{G}}$ satisfying $\prod_g g^{k_g} = \bv{g}$,
\begin{equation}
L_{\bv{g}}\cdot I^{\cZ}_{\bv{k}}(q,z=1)=\iI_{\bv{k}}^{(W,G)}(u=q^{1/r}).
\end{equation}
\end{lemma}
\begin{proof}
Since $ \cD_{q,\bv{g}}^{irr}\cdot\mathbb{I}^{(W,G)}_{\bv{g}}(u=q^{1/r})=0$ and the components of $I^{\cZ}_{\bv{g}}(q,z=1)$ are a complete set of solutions to $\cD_{q,\bv{g}}^{irr}=0$, there exists a unique linear transformation $L_{\bv{g}}:H^*_{\bv{g}}(\cZ) \rightarrow \cH_{\bv{g}}(W, G)$ such that \[L_{\bv{g}}\cdot I^{\cZ}_{\bv{g}}(q,z=1)=\mathbb{I}^{(W,G)}_{\bv{g}}(u=q^{1/r}).\] The proof of the second part of the lemma follows an argument identical the one found in Corollary \ref{corollary_linear_k}, note that equations~\eqref{e:hopeitsnotwrong} and~\eqref{e:remarkeq} still hold in this setting.
\end{proof}
\begin{corollary}\label{Fano asymptotic corollary}
Given $\bv{g} \in \bar{G}$, there exists a unique linear transformation $L_{\bv{g}}:H^*_{\bv{g}}(\cZ) \rightarrow \cH_{\bv{g}}(W, G)$ such that
\[
L_{\bv{g}}\cdot I^{\cZ}_{\bv{g}}(q,z=1)\sim I^{(W,G)}_{\bv{g}}(t=-q^{-1/d}, z=1)\text{ as $q\rightarrow\infty$.} 
\]
\end{corollary}
\begin{proof}
This result easily follows from Equation \ref{Fano asymptotic equation} and Lemma \ref{Lg for LG/Fano}.
\end{proof}
As a consequence of Lemma \ref{Lg for LG/Fano} we obtain the following result:
\begin{theorem}\label{LG/Fano asymptotic correspondence}
Let $\sum_{j=1}^N c_j -d>0$. Then, there exists a unique linear transformation $L^{GW}:H_{CR}^{amb}(\cZ)\longrightarrow \cH_{FJRW}(W, G)$ such that 
\[
L^{GW}\cdot I^{\cZ}(q,\bt,z=1)\sim I^{(W,G)}(t=-q^{-1/d}, \bt,z=1)\quad\text{as $q\rightarrow\infty$}.
\]
\end{theorem}
\begin{proof}
Define $L^{GW}$ to be the block diagonal transformation given by $L^{GW}:=\bigoplus_{g\in\bar{G}}L_{\bv{g}}$. Then the proof is identical to the proof of Theorem \ref{theorem_asymptotic_correspondence} but we use Lemma \ref{Lg for LG/Fano} and Corollary \ref{Fano asymptotic corollary} instead of Corollary \ref{corollary_linear_k}.
\end{proof}
\begin{remark}
One can uniquely recover $ I^{\cZ}(q,\bt,z)$ and $I^{(W,G)}(t, \bt,z)$ from $ I^{\cZ}(q,\bt,1)$ and $I^{(W,G)}(t, \bt,1)$ by using a procedure similar to the one described in Section \ref{z=1}.
\end{remark}
The following result is a natural consequence of Theorem \ref{LG/Fano asymptotic correspondence}: 
\begin{corollary}\label{LG/Fano quantum cohomology}
If $\sum_{j=1}^N c_j -d>0$, the genus zero FJRW theory of the pair $(W,G)$ is completely determined by the by the genus zero Gromov-Witten invariants of $\cZ$. It follows that the FJRW quantum cohomology ring of $(W,G)$ is determined by the quantum cohomology of $\cZ$.
\end{corollary}
\begin{proof}
The genus zero Gromov-Witten invariants of $\cZ$ completely determine the Givental $I$-function $ I^{\cZ}(q,\bt,z)$. As a consequence of Theorem \ref{LG/Fano asymptotic correspondence}, the FJRW $I$-function of the pair $(W,G)$ is uniquely determined by $L^{GW}$ and power series asymptotic expansion at infinity. By means of Birkhoff factorization, the FJRW $I$-function completely determines the FJRW $J$-function and therefore, the genus zero FJRW theory.
\end{proof}
We finish this section by stating a Landau-Ginzburg/general type correspondence. This can be proved in an analogous manner to the Fano case. For the remainder of the section, assume that $\sum_{j=1}^N c_j -d<0$. This condition implies that $\cZ$ is a hypersurface of general type. We then have the following results:
\begin{theorem}\label{LG/gt asymptotic correspondence}
If $\sum_{j=1}^N c_j -d<0$, there exists a unique linear transformation $L^{FJRW}:\cH_{FJRW}(W, G)\longrightarrow H_{CR}^{amb}(\cZ)$ such that 
\[
L^{FJRW}\cdot I^{(W,G)}(t,\bt,z=1)\sim I^{\cZ}(q=t^{-d}, \bt,z=1)\quad\text{as $t\rightarrow\infty$}.
\]
\end{theorem}
\begin{corollary}
If $\sum_{j=1}^N c_j -d<0$, the genus zero GW theory of $\cZ$ is completely determined by the genus zero FJRW theory of the pair $(W,G)$. It follows that the quantum cohomology of $\cZ$ is determined by the FJRW quantum cohomology ring of the pair $(W,G)$.
\end{corollary}
The proofs of these statements are identical to the proofs of Theorems \ref{LG/Fano asymptotic correspondence} and \ref{LG/Fano quantum cohomology} after exchanging the roles of $I^{\cZ}(q,\bt,z)$ and $I^{(W,G)}(t, \bt,z)$.

\bibliographystyle{plain}
\bibliography{references}

\begin{thebibliography}{10}

\bibitem{AGV}
Dan Abramovich, Tom Graber, and Angelo Vistoli.
\newblock Gromov-{W}itten theory of {D}eligne-{M}umford stacks.
\newblock {\em Amer. J. Math.}, 130(5):1337--1398, 2008.

\bibitem{A}
Pedro Acosta.
\newblock Asymptotic {E}xpansion and the {LG}/({F}ano, {G}eneral {T}ype)
  {C}orrespondence.
\newblock arXiv:1411.4162, 2014.

\bibitem{Ad}
Alan Adolphson.
\newblock Hypergeometric functions and rings generated by monomials.
\newblock {\em Duke Math. J.}, 73(2):269--290, 1994.

\bibitem{BCR}
Andrea Brini, Renzo Cavalieri, and Dustin Ross.
\newblock {C}repant resolutions and open strings.
\newblock arXiv:1309.4438, 2013.

\bibitem{BG}
Jim Bryan and Tom Graber.
\newblock The crepant resolution conjecture.
\newblock In {\em Algebraic geometry---{S}eattle 2005. {P}art 1}, volume~80 of
  {\em Proc. Sympos. Pure Math.}, pages 23--42. Amer. Math. Soc., Providence,
  RI, 2009.

\bibitem{ChenR3}
Weimin Chen and Yongbin Ruan.
\newblock Orbifold {G}romov-{W}itten theory.
\newblock In {\em Orbifolds in mathematics and physics ({M}adison, {WI},
  2001)}, volume 310 of {\em Contemp. Math.}, pages 25--85. Amer. Math. Soc.,
  Providence, RI, 2002.

\bibitem{ChenR1}
Weimin Chen and Yongbin Ruan.
\newblock A new cohomology theory of orbifold.
\newblock {\em Comm. Math. Phys.}, 248(1):1--31, 2004.

\bibitem{ChIR}
Alessandro Chiodo, Hiroshi Iritani, and Yongbin Ruan.
\newblock Landau-{G}inzburg/{C}alabi-{Y}au correspondence, global mirror
  symmetry and {O}rlov equivalence.
\newblock {\em Publ. Math. Inst. Hautes \'Etudes Sci.}, 119:127--216, 2014.

\bibitem{ChR}
Alessandro Chiodo and Yongbin Ruan.
\newblock Landau-{G}inzburg/{C}alabi-{Y}au correspondence for quintic
  three-folds via symplectic transformations.
\newblock {\em Invent. Math.}, 182(1):117--165, 2010.

\bibitem{CCIT}
Tom Coates, Alessio Corti, Hiroshi Iritani, and Hsian-Hua Tseng.
\newblock Computing genus-zero twisted {G}romov-{W}itten invariants.
\newblock {\em Duke Math. J.}, 147(3):377--438, 2009.

\bibitem{CCIT2}
Tom Coates, Alessio Corti, Hiroshi Iritani, and Hsian-Hua Tseng.
\newblock A {M}irror {T}heorem for {T}oric {S}tacks.
\newblock arXiv:1310.4163, 2013.

\bibitem{CG}
Tom Coates and Alexander Givental.
\newblock Quantum {R}iemann-{R}och, {L}efschetz and {S}erre.
\newblock {\em Ann. of Math. (2)}, 165(1):15--53, 2007.

\bibitem{CI}
Tom Coates and Hiroshi Iritani.
\newblock A {F}ock {S}heaf {F}or {G}ivental {Q}uantization.
\newblock arXiv:1411.7039, 2014.

\bibitem{CIJ}
Tom Coates, Hiroshi Iritani, and Yunfeng Jiang.
\newblock The {C}repant {T}ransformation {C}onjecture for {T}oric {C}omplete
  {I}ntersections.
\newblock arXiv:1410.0024, 2014.

\bibitem{CIT}
Tom Coates, Hiroshi Iritani, and Hsian-Hua Tseng.
\newblock Wall-crossings in toric {G}romov-{W}itten theory. {I}. {C}repant
  examples.
\newblock {\em Geom. Topol.}, 13(5):2675--2744, 2009.

\bibitem{CR}
Tom Coates and Yongbin Ruan.
\newblock Quantum cohomology and crepant resolutions: a conjecture.
\newblock {\em Ann. Inst. Fourier (Grenoble)}, 63(2):431--478, 2013.

\bibitem{CK}
David~A. Cox and Sheldon Katz.
\newblock {\em Mirror symmetry and algebraic geometry}, volume~68 of {\em
  Mathematical Surveys and Monographs}.
\newblock American Mathematical Society, Providence, RI, 1999.

\bibitem{FJR1}
Huijun Fan, Tyler Jarvis, and Yongbin Ruan.
\newblock The {W}itten equation, mirror symmetry, and quantum singularity
  theory.
\newblock {\em Ann. of Math. (2)}, 178(1):1--106, 2013.

\bibitem{Ga}
Andreas Gathmann.
\newblock Gromov-{W}itten invariants of blow-ups.
\newblock {\em J. Algebraic Geom.}, 10(3):399--432, 2001.

\bibitem{GKZ}
I.~M. Gelfand, A.~V. Zelevinski{\u\i}, and M.~M. Kapranov.
\newblock Hypergeometric functions and toric varieties.
\newblock {\em Funktsional. Anal. i Prilozhen.}, 23(2):12--26, 1989.

\bibitem{GW}
Eduardo Gonzalez and Chris~T. Woodward.
\newblock A wall crossing formula for {G}romov-{W}itten invariants under
  variation of git quotient.
\newblock arXiv:1208.1727, 2012.

\bibitem{GrP}
T.~Graber and R.~Pandharipande.
\newblock Localization of virtual classes.
\newblock {\em Invent. Math.}, 135(2):487--518, 1999.

\bibitem{HH}
Wei~Qiang He and Jian~Xun Hu.
\newblock Orbifold {G}romov-{W}itten invariants of weighted blow-up at smooth
  points.
\newblock {\em Acta Math. Sin. (Engl. Ser.)}, 31(5):825--846, 2015.

\bibitem{Hu}
J.~Hu.
\newblock Gromov-{W}itten invariants of blow-ups along points and curves.
\newblock {\em Math. Z.}, 233(4):709--739, 2000.

\bibitem{ILLW}
Y.~Iwao, Y.-P. Lee, H.-W. Lin, and C.-L. Wang.
\newblock Invariance of {G}romov-{W}itten theory under a simple flop.
\newblock {\em J. Reine Angew. Math.}, 663:67--90, 2012.

\bibitem{YPnotes}
Y.-P. Lee.
\newblock Notes on axiomatic {G}romov-{W}itten theory and applications.
\newblock In {\em Algebraic geometry---{S}eattle 2005. {P}art 1}, volume~80 of
  {\em Proc. Sympos. Pure Math.}, pages 309--323. Amer. Math. Soc., Providence,
  RI, 2009.

\bibitem{LLQW}
Y.-P. Lee, H.-W. Lin, F.~Qu, and C.-L. Wang.
\newblock Invariance of quantum rings under ordinary flops: {III}.
\newblock arXiv:1401.7097, 2014.

\bibitem{LPS}
Y.-P. Lee, Nathan Priddis, and Mark Shoemaker.
\newblock A proof of the {L}andau-{G}inzburg/{C}alabi-{Y}au correspondence via
  the crepant transformation conjecture.
\newblock arXiv:1410.5503, 2014.

\bibitem{LR}
An-Min Li and Yongbin Ruan.
\newblock Symplectic surgery and {G}romov-{W}itten invariants of {C}alabi-{Y}au
  3-folds.
\newblock {\em Invent. Math.}, 145(1):151--218, 2001.

\bibitem{MKH}
A.M. Mathai, Ram~Kishore Saxena, and Hans Haubold.
\newblock {\em The H-Function: Theory and Applications}.
\newblock Springer, 2010.

\bibitem{Miller}
Peter~D. Miller.
\newblock {\em Applied asymptotic analysis}, volume~75 of {\em Graduate Studies
  in Mathematics}.
\newblock American Mathematical Society, Providence, RI, 2006.

\bibitem{PS}
Nathan Priddis and Mark Shoemaker.
\newblock A {L}andau--{G}inzburg/{C}alabi--{Y}au correspondence for the mirror
  quintic.
\newblock arXiv:1309.6262, 2013.

\bibitem{Re}
Miles Reid.
\newblock La correspondance de {M}c{K}ay.
\newblock {\em Ast\'erisque}, (276):53--72, 2002.
\newblock S{\'e}minaire Bourbaki, Vol. 42, 1999/2000.

\bibitem{Ru}
Yongbin Ruan.
\newblock Surgery, quantum cohomology and birational geometry.
\newblock In {\em Northern {C}alifornia {S}ymplectic {G}eometry {S}eminar},
  volume 196 of {\em Amer. Math. Soc. Transl. Ser. 2}, pages 183--198. Amer.
  Math. Soc., Providence, RI, 1999.

\bibitem{Ts}
Hsian-Hua Tseng.
\newblock Orbifold quantum {R}iemann-{R}och, {L}efschetz and {S}erre.
\newblock {\em Geom. Topol.}, 14(1):1--81, 2010.

\bibitem{Zhou}
Jian Zhou.
\newblock {C}repant resolution conjecture in all genera for type {A}
  singularities.
\newblock arXiv:0811.2023, 2008.

\end{thebibliography}

\end{document}